\documentclass{birkjour}
 \newtheorem{thm}{Theorem}[section]
 \newtheorem{cor}[thm]{Corollary}
 \newtheorem{lem}[thm]{Lemma}
 
 \theoremstyle{definition}
 \newtheorem{defn}[thm]{Definition}
 \theoremstyle{remark}
 \newtheorem{rem}[thm]{\textbf{Remark}}
 
 \newtheorem{ex}[thm]{\textbf{Example}}
 \numberwithin{equation}{section}
 \DeclareMathOperator*{\Co}{Conv\,}
 \DeclareMathOperator*{\diam}{Diam\,}

\begin{document}
\title[]{Fixed points for $F$-expanding
mappings in the sense of measures of noncompactness}

\author[Y. Touail]{Youssef Touail$^{*,1}$}
\email{y.touail@usms.ma}
\email{youssef9touail@gmail.com}
\thanks{* Corresponding author.}

\author[A. Jaid]{Amine Jaid$^{1}$}
\email{aminejaid1990@gmail.com}

\vspace{1cm}

\address{$^1$Equipe de Recherche en Math\'{e}matiques Appliqu\'{e}es,\br Technologies de l'Information et de la Communication \br
Facult\'{e} Polydisciplinaire de Khouribga, BP. 25000\br
 Universit\'{e} Sultan Moulay Slimane de Beni-Mellal\br
Morocco}

\author[D. El Moutawakil]{Driss El Moutawakil$^2$}
\email{d.elmotawakil@gmail.com}

\vspace{1cm}

\address{$^{2}$ D\'{e}partement de Math\'{e}matiques, ESEF, Universit\'{e} Chouaib Doukkali, El-Jadida, Morocco}

\date{January 1, 2004}
\dedicatory{}
\vspace{1cm}
\begin{abstract}
In this article, we model with measures of noncompactness the well-known concept of $F$-expanding mappings given by Górnicki (Fixed Point Theory Appl 2017, 9 (2016)). Our results are proved by weakening some assumptions on $F$ and without using the surjectivity. Moreover, the paper contains some examples and an application to the Volterra-integral equation.
\end{abstract}
\maketitle

\section{Introduction}
Due to the importance and fruitful applications of Banach contraction principle \cite{banach}, many authors have generalized
this result by considering some classes of nonlinear mappings which are more
general than contraction mappings
(see \cite{1,15,17} and the references therein).
In 2012, Wardowski \cite{war}
introduced a class of mappings called the $F$-contractions defined as follows:
  \begin{defn}(\cite{war})\label{ward}
 Let $F:\mathbb{R}_+\rightarrow\mathbb{R}$ be a mapping satisfying:\\
$(F1)$ $F$ is strictly increasing,\\
$(F2)$ For each sequence $\{\alpha_n\}$ of positive numbers $\lim\alpha_n=0$ if and only if
$\lim F (\alpha_n)=-\infty$.\\
$(F3)$ There exists $k\in (0, 1)$ such that $\lim_{\alpha\rightarrow0^+}\alpha^kF(\alpha) = 0$.\\
A mapping $T:X\rightarrow X$ is said to be $F$-contraction if there exists $\tau >0$ such that for all $x, y \in X$,
\begin{equation}\label{1}
   d(Tx,Ty)>0\implies \tau+F(d(Tx,Ty))\leq F(d(x,y)).
 \end{equation}
 \end{defn}
An important result is proved in \cite{war} as follows:
 \begin{thm}(\cite{war})
 Let $(X, d)$ be a complete metric space and let $T: X\to X$ be an $F$-contraction.
Then $T$ has a unique fixed point $u\in X$ and for every $x_0\in X$ a sequence
$\{T^nx_0\}_{n\in\mathbb{N}}$ is convergent to $u$.
 \end{thm}
  Following this direction of research, in 2017, Górnicki \cite{g} generalized the well-known notion of expanding mappings by introducing the class of maps called $F$-expanding. He provided a remarkable fixed point theorem for this kind of contractions.
  \begin{defn}(\cite{g})\label{def}
  Let $(X,d)$ be a metric space. A mapping $T : X\to X$ is called $F$-expanding
if there exists $\tau> 0$ such that for all $x, y \in X$,
\begin{equation}\label{}
  d(x, y) > 0 \implies \tau+F(d(x, y))\leq F(d(Tx, Ty)).
\end{equation}
\end{defn}
\begin{thm}(\cite{g})\label{G}
Let $(X, d)$ be a complete metric space and $T : X\to X$ be surjective and $F$-expanding.
Then $T$ has a unique fixed point.
\end{thm}
 On the other hand, in 1930, Kuratowski \cite{kura} was introduced the notion of measure of noncompactness.
This concept has been used by researchers around the world to establish fixed point results in different abstract
spaces. After then, in 1980, Bana\`{s} and Goebel \cite{ban} announced a new axiomatic approach for measures of noncompactness as follows:
\begin{defn}(\cite{ban})\label{msr} A map $\mu: \mathcal{M}_X\rightarrow [0,+\infty[$ is called measure of non-compactness defined on $X$ if it satisfies the following properties:\\
(1) The family $\ker\mu=\{B\in \mathcal{M}_X: \mu(B)=0\}$ is nonempty and $\ker\mu\subset \mathcal{N}_X$,\\
(2) $A\subset B\Rightarrow\mu(A)\leq\mu(B)$,\\
(3) $\mu(B)=\mu(\overline{B})=\mu(\Co(B))$,\\
(4) $\mu(\lambda A+(1-\lambda)B)\leq\lambda\mu(A)+(1-\lambda)\mu(B)$ for all $\lambda\in[0,1]$ and $A$, $B\in \mathcal{M}_X$,\\
(5) if $\{\Omega_n\}$ is a decreasing sequence of nonempty, closed and bounded subsets of $X$ with $\lim\mu(\Omega_n)=0$, then $\Omega_\infty=\cap_n\Omega_n\neq\emptyset$.
\end{defn}
 In Definition \ref{msr}; $X$ is a Banach space, $\mathcal{M}_X$ denotes the collection of all nonempty and bounded subsets of $X$ and
 $\mathcal{N}_X$ denotes the collection of all relatively compact subsets of $X$. Also, $\overline{B}$ and $\Co(B)$ denote the closure and closed convex hull of $B\subset X$, respectively.\\
  Now, let us recall the well-known Schauder fixed point theorem:
\begin{thm}(\cite{schau})\label{sch} Let $C$ be a nonempty, convex and compact subset
of a Banach space $X$. Then, every continuous mapping $T : C\rightarrow C$ has at least one
fixed point.
\end{thm}
In the present paper, we model Definition \ref{ward}, Definition \ref{def} and Theorem \ref{G} in the sense of measures of noncompactness. Namely, we introduce a novel class of mappings $T:C\rightarrow X$ as follows:
\begin{defn}\label{ndef}
 A mapping $T:C\rightarrow X$ is said to be an $F$-expanding in the sense of measures of noncompactness. If there exist $F\in \mathcal{F}$ and $\tau >0$ such that for all $\Omega\subset C$,
\begin{equation}\label{11}
   \mu(\Omega)>0\implies \tau+F(\mu(\Omega))\leq F(\mu(T\Omega)),
 \end{equation}
 \end{defn}
 where $\mathcal{F}$ is the class of mappings $F:\mathbb{R}_+\rightarrow\mathbb{R}$ only satisfying  $(F1)$ and $(F2)$ compared to hypotheses of Theorem \ref{G}.\\
 After that, we prove some new fixed point results for mappings $T: C\to X$ satisfying the condition \eqref{11} without adding the surjectivity. We point out that this hypothesis is used for obtaining a fixed point result in Theorem \ref{G}.\\ Finally, to show the validity of our results in the setting of measures of noncompactness, we give some usable examples and present an existence result for the Volterra-integral equation.
\section{The results}
In this section, we start with the following examples supporting the motivation
of this paper:
\begin{ex}
 Let $F(t) = \ln t$, then $F$ verifies (F1), (F2) and (F3), and condition
\eqref{11} is of the form:
\begin{equation}\label{}
  \mu(T\Omega) \geq e^\tau\mu(\Omega),
\end{equation}
for all $\Omega\subset C$.
\end{ex}
\begin{ex}
 If $F(t) = \ln t + t, t >0$. Clearly, $F$ satisfies (F1), (F2) and (F3).
Every mapping $T: C\to X$ satisfying \eqref{11} is an $F$-expanding map such that:
\begin{equation}\label{}
  \mu(T\Omega)e^{\mu(T\Omega)-\mu(\Omega)}\geq e^\tau\mu(\Omega),
\end{equation}
for all $\Omega\subset C$.
\end{ex}
\begin{ex}
 Let $F(t) = \ln(t^2 + t)$, with $t > 0$. Thus $F$ satisfies (F1), (F2) and (F3), and for
$F$-expanding T, the following condition holds:
\begin{equation}\label{}
  \mu(T\Omega)\frac{\mu(T\Omega)+1}{\mu(\Omega)+1} \geq e^\tau\mu(\Omega),
\end{equation}
for all $\Omega\subset C$.
\end{ex}
Similar functions are, for example, $F(t) = \ln(t^n)$, $n \in \mathbb{N}$, $t > 0$; $F(t) =\ln(\arctan t)$, $t > 0$ and $F(t)=-\frac{1}{t}$, $t>0$ which does not satisfy the condition (F3).\\
The following technical lemmas are crucial to proving Theorem \ref{main}.
 \begin{lem}\label{fr} Let $C$ be a nonempty bounded, closed and convex subset of a Banach space $X$ and $T:C\rightarrow C$ be a continuous selfmapping. If there exist $F\in \mathcal{F}$ and $\tau >0$ such that for all $\Omega\subset C$,
\begin{equation}\label{}
   \mu(T\Omega)>0\implies \tau+F(\mu(T\Omega))\leq F(\mu(\Omega)).
 \end{equation}
  Then $T$ has at least one fixed point and the set of fixed points of $T$ belongs to $\ker\mu$.
 \end{lem}
 \begin{proof}
Consider the sequence of sets $\Omega_0=C, \Omega_{n+1}=\Co T\Omega_n$, $n\geq0$.
If there exists $n_0\in\mathbb{N}$ such that $\mu(\Omega_{n_0})=0$, then $\Omega_0$ is compact, so according to Theorem \ref{sch} $T$ has a fixed point in $C$.\\
Assume that $\mu(\Omega_{n})>0$ for all $n\in\mathbb{N}$. In view of
\begin{equation}
F(\mu(\Omega_{n+1}))=F(\mu(\Co(T\Omega_n)))=F(\mu(T\Omega_n))\leq F(\mu(\Omega_n))-\tau,
\end{equation}
we have
\begin{equation}\label{3}
F(\mu(\Omega_n))\leq F(\mu(\Omega_0))-n\tau.
\end{equation}
From (\ref{3}), we obtain $\lim F(\mu(\Omega_n))=-\infty$ that together with $(F2)$ give:
\begin{equation}\label{nn}
\lim\mu(\Omega_n)=0.
\end{equation}
Since $\Omega_{n+1}\subset \Omega_n$ and each $\Omega_n$ is convex, so $\Omega_\infty=\cap_n\Omega_n$ is convex, thus we have
\begin{equation}
 T\Omega_n\subset T\Omega_{n-1}\subset\Co(T\Omega_{n-1})=\Omega_n.
\end{equation}
Then, $T:\Omega_n\rightarrow \Omega_n$ for all $n\geq0$, which implies that $\Omega_\infty$ is invariant under $T$.
From axiom (5) of Definition \ref{msr}, we obtain that the intersection set $\Omega_\infty$  is a member of the kernel
$\ker\mu$. In fact, since $\mu(\Omega_\infty)\leq\mu(\Omega_n)$ for any $n$, \eqref{nn} implies that $\mu(\Omega_\infty) = 0$. This
yields that $\Omega_\infty\in\ker\mu$. Therefore, Theorem \ref{sch} finishes the proof.
\end{proof}
 \begin{lem}\label{lem}
 Let $C\subset X$ and $T : C \to X$ be a mapping such that $C\subset TC$. Then, there exists a mapping $T' : C \to C$ such that
$TT'$ is the identity map on $C$.
 \end{lem}
 \begin{proof}
 Let $x\in C$, we have $C\subset TC$ implies that $Ty_x=x$ for some $y_x\in C$. Let $T'x = y_x$ for all
$x\in C$. Therefore $(TT')x=T(T'x)=Ty_x=x$ for all $x\in C$.
 \end{proof}
The main result of this paper reads as follows:
 \begin{thm}\label{main}
 Let $C$ be a nonempty bounded, closed and convex subset of a Banach space $X$ and $T:C\rightarrow X$ be a continuous $F$-expanding mapping such that $C\subset TC$.
 Then $T$ has at least one fixed point and the set of fixed points of $T$ belongs to $\ker\mu$.
 \end{thm}
 \begin{proof}
   Let $\Omega\subset C$ such that $\mu(T'\Omega)>0$, where $T': C\to C$ is the mapping defined in Lemma \ref{lem}.
   Denote $\Omega'=T'\Omega$, since $T$ is an $F$-expanding mapping, we obtain
   \begin{equation}\label{x}
    \tau+F(\mu(\Omega'))\leq F(\mu(T\Omega')).
 \end{equation}
  Using Lemma \ref{lem}, we get
  \begin{equation}\label{y}
    T\Omega' = TT'\Omega =\Omega.
  \end{equation}
  Therefore, replacing equality \eqref{y} in the inequality \eqref{x}, we achieve
  \begin{equation}\label{1}
    \tau+F(\mu(T'\Omega))\leq F(\mu(\Omega)).
 \end{equation}
 This means that $T'$ is an $F$-contraction mapping in the sense of measures of noncompactness. So, according to Lemma \ref{fr}, $T'$ has fixed point $u\in C$.
The element $u$ is also a fixed point of $T$ because $T'u = u$ implies that $Tu = T(T'u) = u$.
 \end{proof}
 \begin{ex}\label{example}
 Let $X=\mathbb{R}$ and $C=[0,1]$, define the mapping $T: C\to X$ by
 \begin{equation}\label{}
   Tx=e^{x+1}+ex-e,
 \end{equation}
 for all $x\in[0,1]$. The mapping $T$ satisfies the inequality
 \begin{equation}\label{}
   |Tx-Ty|\geq e|x-y|,
 \end{equation}
 for all $x,y\in [0,1]$.\\
 Let $\mu :\mathcal{M}_X \to \mathbb{R}^+$ the measure of noncompactness of the diameter (see \cite{ban}) defined by
\begin{equation}\label{}
  \mu(\Omega) = \diam (\Omega):= sup_{x,y\in\Omega}|x-y|.
\end{equation}
Thus
 \begin{equation}\label{}
       \mu(T\Omega)\geq e\mu(\Omega),
     \end{equation}
for all $\Omega\subset C$.\\
Therefore, $T$ is an $F$-expanding mapping with $\tau=1$ and $F(t)=\ln t, t>0$.\\
Furthermore, we have
\begin{equation}\label{}
  [0,1]\subset T([0,1])=[0,e^2].
\end{equation}
 Therefore all conditions of Theorem \ref{main} are satisfied and $0=T0$.
 \end{ex}
 \begin{rem}
As a comparison with Theorem \ref{G}, Example \ref{example} shows the importance of the proven Theorem \ref{main}. We note that the surjectivity is omitted.
 \end{rem}
Using Theorem \ref{main}, we establish the following result to the classical metric
fixed point theory.
\begin{cor}
Let $C$ be a nonempty, bounded, closed and convex subset of a Banach
space $X$ and suppose that $T : C \to X$ is a continuous
mapping such that $C\subset TC$. If there exists $\tau>0$ and $F\in\mathcal{F}$ such that
\begin{equation}\label{kk}
\tau + F(||x-y||\leq F(||Tx-Ty||),
\end{equation}
for all $\Omega\subset C$ and $x,y\in\Omega$.
Then $T$ has a fixed point in $C$.
\end{cor}
\begin{proof}

Applying the definition of $\mu=\diam$ to \eqref{kk}, we obtain
\begin{equation}\label{}
\tau+F(\mu(\Omega))\leq\mu(F(T\Omega)),
\end{equation}
for all $\Omega\subset C$.\\
So, from Theorem \ref{main} we get the desired result.
\end{proof}
\begin{rem}
The above theorem illustrates the situation where a measure of noncompactness explicitly derives from a norm.
\end{rem}

\section{Application}
In this section, we study the existence of solutions for Volterra-integral equation. For this purpose, let $X=\mathcal{C}([0,1],\mathbb{R})$ the space of all continuous functions from $[0,1]$ into $\mathbb{R}$. Denote $C$ the space of all functions $x$ belonging to $X$ such that for all $t\in[0,1]$, $|x(t)|\leq\xi$ for some $\xi>0$. Note that $X$ is a Banach space by considering the standard norm:
\begin{equation}
 ||x|| = \max_{t\in[0,1]}|x(t)|.
\end{equation}
Consider the Volterra-integral equation:
\begin{align}\label{4}
x(t) &= \int_{0}^{t}k(s,x(s))ds,
\end{align}
where $x\in C$ and $k:[0,1]\times \mathbb{R}\rightarrow \mathbb{R}$ is a continuous mapping. Similar works in this direction can be found in \cite{t1,t2,orth,t3,t4,t5}.\\
Define the operator $T: C\to X$ as follows:
\begin{equation}
  T(x)(t)=\int_{0}^{t}k(s,x(s))ds.
\end{equation}
Then, (\ref{4}) has a solution if and only if $T$ has at least one fixed point in $C$.\\
Under the above assumptions, we are ready to assert the following theorem:
\begin{thm}
If there exists $\tau>0$ such that
\begin{equation}\label{p}
  |x(t)|<\frac{1}{\tau},
\end{equation}
 and
\begin{equation}\label{i}
  x(t)\ne0\implies|T(x)(t)|\geq \frac{|x(t)|}{1-\tau|x(t)|},
  \end{equation}
for all $\Omega\subset C$, $x\in \Omega$ and $t\in[0,1]$. Then, the nonlinear integral equation \eqref{4} has a solution in $C$.
\end{thm}
\begin{proof}
Let $\Omega\subset C$, $x\in\Omega$ and $t\in[0,1]$. Suppose that \eqref{p} and \eqref{i} are fulfilled for some $\tau>0$.\\
   Then, we obtain
\begin{equation}\label{}
  \tau-\frac{1}{|x(t)|}\leq-\frac{1}{|T(t)|},
\end{equation}
which implies that
\begin{equation}\label{}
  \tau+F(|x(t)|)\leq F(|T(t)|),
\end{equation}
where $F(t)=-\frac{1}{t}$ for all $t>0$.\\
Taking the measure of noncompactness of the norm (see \cite{ban}):
  \begin{equation}\label{norm}
    \mu(\Omega):=\sup_{x\in\Omega}||x||,
  \end{equation}
  then, we get
  \begin{equation}\label{}
  \tau+F(\mu(\Omega))\leq F(\mu (T\Omega)).
\end{equation}
This means that $T$ is an $F$-expanding mapping in the sense of measures of noncompactness. According to Theorem \ref{main}, we deduce that (\ref{4}) has a solution in $C$.
\end{proof}

\section*{Data Availability}
No data were used to support this study.
\section*{Conflicts of Interest}
The authors declare that they have no conflicts of interest.

\end{document}